\documentclass[10pt]{amsart}
\usepackage{amsmath,amsthm,amssymb, amscd, amsfonts}

\theoremstyle{plain}
\newtheorem{theorem}{Theorem}[section]
\newtheorem{corollary}[theorem]{Corollary}

\newtheorem{proposition}[theorem]{Proposition}
\newtheorem{lemma}[theorem]{Lemma}

\theoremstyle{definition}

\theoremstyle{remark}
\newtheorem{remark}[theorem]{Remark}

\numberwithin{equation}{section}\theoremstyle{plain}

\newcommand{\C}{{\mathcal C}}

\newcommand\id{\operatorname{id}}

\newcommand\Aut{\operatorname{Aut}}
\newcommand\Out{\operatorname{Out}}

\newcommand\cop{\operatorname{cop}}

\newcommand\Opext{\operatorname{Opext}}

\newcommand\res{\operatorname{res}}
\newcommand\Ind{\operatorname{Ind}}
\newcommand\op{\operatorname{op}}

\newcommand\Rep{\operatorname{Rep}}

\newcommand\vect{\operatorname{Vec}}

\begin{document}
\title[]{On quasitriangular structures in Hopf algebras arising from exact group factorizations}
\author{Sonia Natale}
\address{Facultad de Matem\'atica, Astronom\'\i a y F\'\i sica.
Universidad Nacional de C\'ordoba. CIEM -- CONICET. (5000) Ciudad
Universitaria. C\'ordoba, Argentina}
\email{natale@famaf.unc.edu.ar
\newline \indent \emph{URL:}\/ http://www.famaf.unc.edu.ar/$\sim$natale}
\begin{abstract} We show that  bicrossed product Hopf algebras arising
from exact factorizations in almost simple finite groups, so in
particular, in simple and symmetric groups, admit no
quasitriangular structure.
 \end{abstract}
\thanks{Partially supported by  CONICET,
SeCYT (UNC) and  FAMAF (Rep\'ublica Argentina).}


\date{September 1, 2010.}

\maketitle

\section{Introduction and main result}

We shall work over an algebraically closed field $k$ of
characteristic zero. Let $H$ be a finite dimensional  Hopf
algebra. The category $\Rep H$ of its finite dimensional
representations is a \emph{finite tensor category} with tensor
product induced by the comultiplication of $H$ and unit object
$k$. When $H$ is semisimple, $\Rep H$ is a fusion category over
$k$, that is, a semisimple finite tensor category over $k$.

Tensor categories of the form $\Rep H$ are characterized by being endowed with a fiber functor to the category $\vect_k$ of vector
spaces over $k$. The forgetful functor $\Rep H \to \vect_k$ is a
fiber functor and other fiber functors correspond to
\emph{twisting} the comultiplication of $H$.

\medbreak An important class of Hopf algebras over $k$ is that of
quasitriangular Hopf algebras, introduced by Drinfeld in
\cite{drinfeld}. These are Hopf algebras $H$ endowed with a so
called $R$-matrix $R \in H \otimes H$. When $(H, R)$ is
quasitriangular, the category $\Rep H$ is a braided category over
$k$ with respect to the braiding $c: U \otimes V \to V \otimes U$
induced by the action of the $R$-matrix, for representations $U,
V \in \Rep H$.

The main feature of quasitriangular Hopf algebras is that they
give rise to universal solutions of the Quantum Yang-Baxter
Equation. They are also related to the construction of
invariants of knots and $3$-manifolds. See  \cite{kassel}.

\medbreak In this paper we consider the question of the existence
of quasitriangular structures in a class of semisimple Hopf
algebras arising from exact factorizations of finite groups. These
are amongst the first examples of non-commutative
non-cocommutative Hopf algebras. They were first
discovered and studied by G. I. Kac in the late 60's.

Suppose that $G = F\Gamma$ is an exact factorization of the finite
group $G$, into its subgroups $\Gamma$ and $F$, such that $F \cap \Gamma = 1$.  Equivalently, $F$
and $\Gamma$ form a matched pair of finite groups with the actions
$\vartriangleleft : \Gamma \times F \to \Gamma$,
$\vartriangleright : \Gamma \times F \to F$, defined by \begin{equation*}sx = (x
\vartriangleleft s)(x \vartriangleright s),\end{equation*} $x \in F$, $s \in
\Gamma$. We shall call an exact factorization \emph{proper} if $F$ and $\Gamma$ are proper subgroups of $G$.

Associated to this exact factorization and appropriate cohomology
data $\sigma$ and $\tau$, there is a semisimple bicrossed product
Hopf algebra $H = k^\Gamma \, {}^{\tau}\#_{\sigma}kF$. The Hopf
algebra $H$ fits canonically into an \emph{abelian} exact sequence
of Hopf algebras \begin{equation}\label{sec-abel}k \to k^{\Gamma}
\to H \to kF \to k.\end{equation} Moreover, every Hopf algebra
fitting into such exact sequence can be described in this way.
This gives a bijective correspondence between the equivalence
classes of  Hopf algebra extensions \eqref{sec-abel} and a certain
abelian group $\Opext (k^\Gamma, kF)$ associated to the matched
pair $(F, \Gamma)$. We refer the reader to \cite{ma-ext,
ma-newdir} for the notion of abelian exact sequence and the
cohomology theory underlying it.

In the case of bicrossed products $H = k^\Gamma \# kF$, the
classification of so-called \emph{positive} quasitriangular
structures in $H$ appears in the paper \cite{LYZ}. These are
related to set-theoretical solutions of Yang-Baxter equation.

\medbreak The main result of this note is the following.

\begin{theorem}\label{main} Let $G = F\Gamma$ be a proper exact factorization of the finite group $G$.
Let also $H$ be a Hopf algebra as in \eqref{sec-abel} associated to this factorization. Assume $G$ is an almost simple group. Then $H$ admits no quasitriangular structure.
 \end{theorem}

Recall that a finite group $G$ is called \emph{almost simple} if
$G$ is isomorphic to a group $\tilde G$ such that $N \leq \tilde G
\leq \Aut N$, for some non-abelian finite simple group $N$. In
particular, the following are almost simple groups:

(a) $G$ a finite simple group, and

(b) $G = \mathbb S_n$ is the symmetric group on $n$ symbols, $n
\geq 5$.

Hence the theorem implies that any bicrossed product Hopf algebra associated to an exact factorization in any group in one of the classes (a) or (b) admits no quasitriangular structure.

\medbreak There is a vast literature on the classification of
factorizations in finite and, in particular, finite simple and
almost simple groups. Group factorizations are related to the
problem of determining the finite primitive permutation groups
with a regular subgroup.  We mention, in particular, the
references \cite{wiegold-williamson} for the cases of alternating
and symmetric groups,  \cite{giudici} for a determination of all
exact factorizations in sporadic simple groups, and \cite{LPS-1,
LPS} for all the maximal factorizations of simple and almost
simple groups.

\medbreak Theorem \ref{main} will be proved in Sections
\ref{simple} and \ref{almost-simple}.  See Theorems \ref{simp} and \ref{alm-simp}. We treat separately the case of simple groups in Section \ref{simple}.
In the case of general almost simple groups, we make use in our proof of the classification of finite simple groups.
We point out that this is not needed in the particular case of symmetric groups, which is stated in Proposition \ref{sym-2}.

Our proof relies on the classification of full fusion
subcategories of the categories of representations of (twisted)
quantum doubles of a finite group, due to D. Naidu, D. Nikshych
and S. Witherspoon \cite{NNW}; indeed, the results in \cite{NNW}
amount to a description of all group-theoretical braided fusion
categories.

The Drinfeld double of a Hopf algebra $H$ as in \eqref{sec-abel}
is twist equivalent to a twisted Drinfeld double of $G$, as
introduced by Dijkgraaf, Pasquier and Roche \cite{DPR}, in view of
the description of the Drinfeld double of $H$ given in \cite{MBG,
gp-ttic}. On the other hand, a quasitriangular structure in $H$
provides a canonical embedding of $\Rep H$ into $\Rep D(H)$; so
that the category $\Rep H$ must coincide with one of those
appearing in the classification result of \cite{NNW}. This would
also be a fruitful approach in classifying quasitriangular
structures in any given abelian extension or, furthermore, in a
group-theoretical Hopf algebra in general.

\medbreak \textbf{Acknowledgement.} These results are the outgrowth of some questions that were
brought to the author's attention by S. Montgomery. The author thanks her for sharing her insights and for interesting references and discussions.

\section{Preliminaries}

\subsection{Finite dimensional Hopf algebras and their representations}

Let $H$ be a finite dimensional Hopf algebra over $k$. Recall that a \emph{twist} in $H$ is an invertible element $J\in
H\otimes H$ such that
\begin{align}
\label{deltaj}(\Delta\otimes \id)(J)(J\otimes 1)=&(\id\otimes\Delta)(J)(1\otimes J),\\
(\epsilon\otimes \id)(J)=&(\id\otimes\epsilon)(J)=1.
\end{align}
If $J\in H\otimes H$ is a twist, then there is a new Hopf algebra
$(H^{J}, \Delta^J, \mathcal S^J)$, where $H^J = H$ as algebras,
with comultiplication $\Delta^J(h)=J^{-1}\Delta(h)J$, and antipode
$\mathcal S^J(h)= v^{-1}\mathcal S(h)v$, $\forall h\in H$, with $v=m
(\mathcal S\otimes \id)(J)$.

The Hopf algebras $H$ and $H'$ are called \emph{twist equivalent} if
$H'\simeq H^J$. It is known that $H$ and $H'$ are twist equivalent
if and only if $\Rep H \simeq \Rep H'$ as tensor categories
\cite{scha}. Therefore, properties like being semisimple or
(quasi)triangular, are preserved under twisting deformations.

\begin{remark}\label{tens-subc} Full tensor subcategories of $\Rep H$
correspond, via tannaka duality arguments, to quotient Hopf
algebras of $H$. That is, every full tensor subcategory of $\Rep
H$ is of the form $\Rep L$ where $\pi: H \to L$ is a quotient Hopf
algebra and the inclusion $\Rep L \to \Rep H$ is given by
restriction of representations along $\pi$.\end{remark}

Suppose $H$ and $H'$ are twist equivalent. Then $\Rep H \simeq
\Rep H'$ as tensor categories and therefore there is a bijective
correspondence between quotient Hopf algebras of $H$ and $H'$.

Indeed, if $\pi:H \to L$ is a Hopf algebra map and $J\in H\otimes
H$ is a twist, then $(\pi\otimes\pi)(J)$ is a twist for $L$ and
$\pi: H^J\to L^{(\pi\otimes \pi)(J)}$ is a Hopf algebra map. This
establishes the correspondence mentioned above.

\subsection{Quasitriangular structures}

Recall that a quasitriangular structure in Hopf algebra $H$
consists of the data of an invertible element $R \in H \otimes H$,
called an \emph{$R$-matrix}, satisfying:

\begin{itemize}\item[(QT1)] $(\Delta \otimes \id)(R) = R_{13}R_{23}$.

\item[(QT2)] $(\epsilon \otimes \id)(R) = 1$.

\item[(QT3)] $(\id \otimes \Delta)(R) = R_{13}R_{12}$.

\item[(QT4)] $(\id \otimes \epsilon)(R) = 1$.

\item[(QT5)] $\Delta^{\cop}(h) = R \Delta(h) R^{-1}$, $\forall h \in H$.
\end{itemize}

In that case, $(H, R)$ is called a
quasitriangular Hopf algebra. Suppose $R \in H \otimes H$ is an
$R$-matrix. Then there are Hopf algebra maps $f_R: {H^*}^{\cop}
\to H$ and $f_{R_{21}}:H^* \to H^{\op}$ given, respectively, by
$$f_R(p) = \langle p, R^{(1)} \rangle R^{(2)},
\quad f_{R_{21}}(p) = \langle p, R^{(2)} \rangle R^{(1)},$$ for
all $p \in H^*$, where we use the notation $R =
R^{(1)} \otimes R^{(2)} \in H \otimes H$.

\medbreak Recall that for a finite dimensional Hopf algebra $H$, its Drinfeld
double, $D(H)$, is a quasitriangular Hopf algebra containing $H$ and $H^{*\cop}$ as Hopf subalgebras. We have $D(H) = H^{* \cop}
\otimes H$ as a coalgebra, with canonical $R$-matrix $\mathcal R
= \sum_i h^i \otimes h_i$, where $(h_i)_i$ is a basis of $H$ and
$(h^i)_i$ is the dual basis.

\begin{remark}\label{dH} Suppose $(H, R)$ is a finite dimensional quasitriangular
Hopf algebra, and let $D(H)$ be the Drinfeld double of $H$.  Then
there is a surjective Hopf algebra map $f: D(H) \to H$, such that $(f\otimes f)
\mathcal R = R$. The map $f$ is determined by $f(p\otimes h) = f_R(p)h$, for all $p \in H^*$, $h \in H$. \end{remark}

Indeed, we have $\Rep D(H) = \mathcal Z(\Rep H)$: the center of
the tensor category $\Rep  H$, and the map $f$ corresponds to the
canonical inclusion of the braided tensor category $\Rep H$ (with
braiding determined by the action of the $R$-matrix) into its
center.

\subsection{Abelian extensions and factorizable
groups}\label{ab-ext}

Let $(F, \Gamma)$ be a matched pair of finite groups. That is, $F$
and $\Gamma$ are endowed  with actions by permutations $\Gamma
\overset{\vartriangleleft}\leftarrow \Gamma \times F
\overset{\vartriangleright}\to F$ such that
\begin{align}\label{comp1}
s \vartriangleright xy & = (s \vartriangleright x) ((s \vartriangleleft x) \vartriangleright y), \\
\label{comp2} st \vartriangleleft x & = (s \vartriangleleft (t
\vartriangleright x)) (t \vartriangleleft x),
\end{align} for all $s, t \in \Gamma$, $x, y \in F$.

Given finite groups $F$ and $\Gamma$, providing them with a pair
of compatible actions is equivalent to giving a group $G$ together
with an exact factorization $G = F \Gamma$: the actions
$\vartriangleright$ and $\vartriangleleft$ are determined by the
relations $gx = (g \vartriangleright x)(g \vartriangleleft x)$, $x
\in F$, $g \in \Gamma$.

\medbreak Consider the left action of $F$ on $k^\Gamma$, $x.
\phi(g) = \phi(g \vartriangleleft x)$, $\phi \in k^\Gamma$, and
let $\sigma: F \times F \to (k^{\times})^\Gamma$ be a normalized
2-cocycle. Dually, we consider the right action of $\Gamma$ on $k^F$,
$\psi(x).g = \psi(x \vartriangleright g)$, $\psi \in k^F$, and let
$\tau: F \times F \to
(k^{\times})^\Gamma$ be a normalized 2-cocycle.

Assume in addition that $\sigma$ and  $\tau$ obey  the
following compatibility conditions:
\begin{align*} & \sigma_{ts}(x, y) \tau_{xy}(t, s)  =
 \tau_x(t, s) \, \tau_y(t \vartriangleleft (s \vartriangleright x), s \vartriangleleft x) \,
\sigma_{t}(s \vartriangleright x, (s \vartriangleleft x) \vartriangleright y) \, \sigma_{s}(x, y), \\
&\label{norm2-sigma-tau} \sigma_1(s, t)  = 1,  \qquad \tau_1(x, y)
= 1, \end{align*} for all $x, y \in F$, $s, t \in \Gamma$, where $\sigma = \sum_{s \in \Gamma}\sigma_s e_s$, $\tau = \sum_{x\in F}\tau_x e_x$,  $e_s \in k^{\Gamma}$ and $e_x \in k^F$ being the canonical idempotents.

The vector space $H = k^\Gamma \otimes k F$  becomes a
(semisimple) Hopf algebra with the crossed product algebra
structure and the crossed coproduct coalgebra structure. We shall
use the notation $H = k^\Gamma \, {}^{\tau}\#_{\sigma}kF$. The
multiplication and comultiplication of $H$ are determined by the formulas
\begin{align*} (e_g \# x)(e_h \# y) & = e_{g \vartriangleleft x, h}\, \sigma_g(x,
y) e_g \# xy, \\
 \Delta(e_g \# x) & = \sum_{st=g} \tau_x(s, t)\,
e_s \# (t \vartriangleright x) \otimes e_{t}\# x,
\end{align*}
for all $g,h\in \Gamma$, $x, y\in F$.

\medbreak There is an exact sequence of Hopf algebras $$1 \to
k^\Gamma \to H \to kF \to 1,$$ also called an \emph{abelian} exact
sequence. Conversely, every Hopf algebra $H$ fitting into an exact
sequence of this form is isomorphic to $k^\Gamma \,
{}^{\tau}\#_{\sigma}kF$ for appropriate actions
$\vartriangleright$, $\vartriangleleft$, and compatible cocycles
$\sigma$ and $\tau$. See \cite{kac, majid-ext, t-ext, ma-ext,
ma-newdir}.

\begin{remark}\label{div-f} Let $H = k^\Gamma \, {}^{\tau}\#_{\sigma}kF$ as above. For $s \in \Gamma$, let $F^s \subseteq F$ be the stabilizer of $s$ under the action of $F$. Then the map $\sigma_s:F^s \times F^s \to k^{\times}$ defines a 2-cocycle on $F^s$. For each representation $\rho$ of the twisted group algebra $k_{\sigma_s}F^s$, consider the induced representation $V_{s, \rho} = \Ind_{k^{\Gamma} \#_{\sigma}kF^s}^H s \otimes \rho$.

The irreducible representations of $H$ are classified by $V_{s, \rho}$, where $s$ runs over a set of representatives of the orbits of $F$ in $\Gamma$ and $\rho$ runs over the isomorphism classes of  irreducible representations of $k_{\sigma_s}F^s$ \cite{MoW}.

Note that $\dim V_{s, \rho} = [F: F^s] \, \deg \rho$. So, in particular, for all irreducible representation $V$ of $H$ we have that $\dim V$ divides the order of $F$. \end{remark}

\subsubsection{Kac exact sequence} \label{kac-es} Fix a matched pair of groups $(F, \Gamma)$.
The set of equivalence classes of extensions $1 \to k^\Gamma \to H
\to kF \to 1$ giving rise to these actions is denoted by
$\Opext(k^\Gamma, kF)$: it is a finite group under the Baer
product of extensions.

\bigbreak The class of an element of  $\Opext(k^\Gamma, kF)$ can
be represented by a pair of compatible cocycles $(\tau, \sigma)$. The group $\Opext
(k^\Gamma, kF)$ can also be described as the first cohomology group of a
certain double complex \cite[Proposition 5.2]{ma-ext}.

An important result of G. I. Kac \cite{kac, ma-newdir},
says that there is a long exact sequence
\begin{align*}
0 & \to H^1(G, k^{\times}) \xrightarrow{\res}   H^1(F, k^{\times})
\oplus  H^1(\Gamma, k^{\times}) \to \Aut(k^\Gamma \# kF)
\\ & \to  H^2(G, k^{\times}) \xrightarrow{\res}  H^2(F, k^{\times}) \oplus  H^2(\Gamma, k^{\times})
 \to \Opext(k^\Gamma, kF) \\ & \xrightarrow{\kappa}  H^3(G,
k^{\times}) \xrightarrow{\res}  H^3(F, k^{\times}) \oplus
H^3(\Gamma, k^{\times}) \to \dots
\end{align*}

This result turns out to be an important tool in calculations related to the Opext group.
In this paper we shall use the following result, contained in \cite[Theorem 1.3]{gp-ttic}, concerning the map $\kappa: \Opext(k^\Gamma, kF)  \to  H^3(G,
k^{\times})$:

\begin{theorem}\label{kac} The Drinfeld double $D(H)$ is twist equivalent to the twisted quantum double $D^{\omega}(G)$, where $\omega = \kappa(\tau, \sigma)$ is the 3-cocycle arising from the Kac exact sequence.
In other words, there is an equivalence of fusion categories $\Rep D(H) \simeq \Rep
D^{\omega}(G)$.  \end{theorem}

The twisted quantum double $D^{\omega}(G)$, $\omega \in H^3(G, k^{\times})$, is the quasi-Hopf algebra introduced by Dijkgraaf, Pasquier and Roche \cite{DPR}.
For the case of split extensions, that is when $(\tau, \sigma) = 1$ and hence $\omega = 1$, this result was obtained previously in \cite{MBG}.

\begin{corollary}\label{cor-qt} Let $G = F\Gamma$ be an exact factorization of the group $G$. Let also $H$ be a Hopf algebra fitting into an abelian extension $k \to k^{\Gamma} \to H \to kF \to k$ corresponding to this factorization. Suppose in addition that $H$ admits a quasitriangular structure. Then $\Rep H$ is equivalent to a  full (braided) fusion subcategory of $\Rep D^{\omega}(G)$.  \end{corollary}

\begin{proof} It follows from Theorem \ref{kac}, in view of Remarks \ref{tens-subc} and \ref{dH}. \end{proof}

\subsection{Fusion subcategories of $D^{\omega}(G)$}\label{nnw}

Let $G$ be a finite group and let $\omega \in H^3(G, k^{\times})$.  Isomorphism classes of simple objects in $\Rep D^{\omega}(G)$ are parameterized by pairs $(g, \pi)$, where $g$ runs over a set of representatives of the conjugacy classes of $G$, and $\pi$ is an irreducible representation of the twisted group algebra $k_{\beta_g}C_G(g)$, where the 2-cocycle $\beta_g: C_G(g) \times C_G(g) \to k^{\times}$ is given by \begin{equation*}\beta_g(x, y) = \dfrac{\omega(g, x, y) \omega(x, y, g)}{\omega(x, g, y)}.  \end{equation*}

In view of the description of irreducible representations in crossed products given in \cite{MoW}, this can be seen as a consequence of the fact that, as an algebra, $D^{\omega}(G)$ is a crossed product $D^{\omega}(G) = k^G \#_{\beta}kG$ with respect to the adjoint action of $G$ on itself and the 2-cocycle $\beta: kG \otimes kG \to k^{G}$  determined by $$\beta_g(x, y) = \dfrac{\omega(g, x, y) \omega(x, y, y^{-1}x^{-1}gxy)}{\omega(x, x^{-1}gx, y)}.$$
The irreducible representation corresponding to the pair  $(g, \pi)$ may thus be identified with the induced representation $V_{(g, \pi)} = \Ind_{k^G\#_{\beta}kC_G(g)}^{D^{\omega}(G)} g \otimes \pi$. Its dimension is therefore \begin{equation}\label{dim-v}\dim V_{(g, \pi)} = [G: C_G(g)] \, \deg \pi.\end{equation}
See \cite{ACM}.

\medbreak Full fusion subcategories of $\Rep D^{\omega}(G)$ are
determined in \cite[Theorem 5.11]{NNW}. They are in bijection with
tripes $(K_1, K_2, B)$ where
\begin{itemize}\item[(a)] $K_1$ and $K_2$ are normal subgroups of $G$
that centralize each other, and
\item[(b)] $B : K_1 \times K_2 \to k^{\times}$ is
a $G$-invariant $\omega$-bicharacter, that is, it satisfies
\begin{equation*}B(x, yz) = \beta_x^{-1} (y, z)B(x, y)B(x, z), \quad B(tx, y) = \beta_y(t, x)B(t, y)B(x, y), \end{equation*} for all $t, x \in K_1$, $y, z \in K_2$, and
the invariance condition
\begin{equation*}B(x^{-1}ax, h) = \dfrac{\beta_a(x, h)\beta_a(xh, x^{-1})}{\beta_a(x, x^{-1})} \, B(a, xhx^{-1}), \end{equation*}
for all $x, y \in G$, $h \in K_2$, $a \in K_1$. See \cite[Definition 5.4]{NNW}. \end{itemize}

\medbreak The bijection is determined as follows. The triple $(K_1, K_2, B)$ gives rise to the fusion subcategory $\C(K_1, K_2, B)$, defined as the full abelian subcategory of $\Rep D^{\omega}(G)$ generated by the simple objects $V_{(g, \pi)}$, where $g$ runs over a set of representatives of conjugacy classes of $G$ in $K_1$, and $\pi$ is an irreducible $\beta_g$-character of $C_G(g)$ such that $\pi(h) = B(g, h) \deg \pi$, for all $h \in K_2$.

\medbreak In the other direction, a full fusion subcategory $\C$ of $\Rep D^{\omega}(G)$ determines a triple $(K^{\C}_1, K^{\C}_2, B^{\C})$, where
\begin{align*}K^{\C}_1 &= \{ gag^{-1} | \, g \in G, \, (a, \pi) \in \C \,\text{ for some } \pi \}, \\ K_2^{\C} & = \bigcap_{\pi: \, (e, \pi) \in \C} \ker \pi. \end{align*}
(An equivalent definition of $K_2^{\C}$ consists in letting  $\C \cap \Rep G \simeq \Rep G/K_2^{\C}$.) The $\omega$-bicharacter $B^{\C}:  K_1^{\C} \times  K_2^{\C} \to k$ is determined by
\begin{equation*} B^{\C}(x^{-1}gx, h) := \dfrac{\beta_g(x, h)\beta_g(xh, x^{-1}) \, \pi(xhx^{-1})}{\beta_g(x, x^{-1}) \, \deg \pi},\end{equation*} where $(g, \pi)$ is a class of an irreducible representation of $D^{\omega}(G)$ that belongs to $\C$.

The above assignments are inverse bijections.

\medbreak By \cite[Lemma 5.9]{NNW} the dimension of the fusion
subcategory $\C(K_1, K_2, B)$ corresponding to the triple $(K_1, K_2, B)$
is $|K_1|[G : K_2]$. Hence $\dim \C(K_1, K_2, B) = |G|$ if and only if
$|K_1| = |K_2|$.

It follows from the description in \cite[(25)]{NNW} and the
definition of $G$-invariant $\omega$-bicharacter, that  $\C(1, 1, 1) \simeq \Rep G$.

\medbreak As a special case,  full  fusion subcategories of $\Rep D(G)$ are parameterized by triples
$(K_1, K_2, B)$, where
\begin{itemize}\item[(a')] $K_2, K_1$ are normal subgroups of $G$ centralizing each other; \item[(b')] $B: K_1 \times K_2 \to k^{\times}$ is a
$G$-invariant bicharacter. \end{itemize} See \cite[Theorem
3.12]{NNW}. Let  $\C(K_1, K_2, B)$ be the fusion subcategory of $D(G)$
corresponding to the triple $(K_1, K_2, B)$. The dimension of $\C(K_1,
K_2, B)$ is $|K_1| |G: K_2|$ \cite[Lemma 3.10]{NNW}.

\medbreak By Corollary \ref{cor-qt}, every quasitriangular structure on a bicrossed product Hopf algebra $H = k^{\Gamma} {}^{\tau}\#_{\sigma}kF$ as in \eqref{sec-abel} determines a full fusion subcategory of $\Rep D^{\omega}(G)$, where $G = F \Gamma$ and $\omega = \kappa(\tau, \sigma)$ is the 3-cocycle coming from $H$ in the Kac exact sequence. Hence we get:

\begin{corollary}\label{h-triple} Let $H = k^{\Gamma} {}^{\tau}\#_{\sigma}kF$ be an abelian extension as in \eqref{sec-abel}.  Then a quasitriangular structure in $H$ determines a triple $(K_1, K_2, B)$ satisfying (a) and (b) with respect to  $G = F \Gamma$ and $\omega = \kappa(\tau, \sigma)$,  such that $|K_1| = |K_2|$. \qed \end{corollary}

\medbreak The following proposition gives some restrictions on the possible triples in terms of the involved factorization.

\begin{proposition} Let $G = F\Gamma$ be an exact factorization of the group $G$. Let also $H$ be a Hopf algebra extension \eqref{sec-abel} corresponding to this factorization and assume that $H$ admits a quasitriangular structure.

Let $(K_1, K_2, B)$ be the triple determined by the quasitriangular structure in $H$.  Then for all $g \in K_1$, and for all irreducible character $\pi$ of the twisted group algebra $k_{\beta_g}C_G(g)$ such that $\pi(h) = B(g, h)\deg \pi$, for all $h \in K_2$, the product $[G: C_G(g)] \deg \pi$ divides the order of $F$. \end{proposition}

\begin{proof} The group $K_1$ is determined as the union of the conjugacy classes of
elements  $g \in G$ such that the (class of the) irreducible
representation corresponding to some pair $(g, \pi)$ belongs to
the fusion subcategory $\Rep H = \C(K_1, K_2, B)$ of
$D^{\omega}(G)$ \cite[5.2]{NNW}. By the description of the
irreducible representations belonging to $\C(K_1, K_2, B)$, these
are exactly those corresponding to pairs $(g, \pi)$ where $g \in
K_1$ and $\pi$ is an irreducible character of the twisted group
algebra $k_{\beta_g}C_G(g)$, such that $\pi(h) = B(g, h)\deg \pi$,
for all $h \in K_2$.

The proposition follows from Remark \ref{div-f} and Formula \eqref{dim-v} for the dimension of such an irreducible representation. \end{proof}

\section{Some examples associated to symmetric
groups}\label{symm}

Let $C_n$ be the cyclic group of order $n$ and $\mathbb S_n$ the symmetric group on $n$ symbols.
Let $H = k^{C_n} \# k\mathbb S_{n-1}$ be the bismash product
(split abelian extension) associated to the matched pair $(C_n,
\mathbb S_{n-1})$ arising from the exact factorization $\mathbb S_{n} = C_n \mathbb S_{n-1}$.

In this case, the group $\Opext(k^{C_n}, k\mathbb S_{n-1})$
corresponding to this matched pair is trivial \cite[Theorem 4.1]{ma-calculations}. Hence $H$ is the
unique, up to isomorphisms, Hopf algebra fitting into an abelian
exact sequence $k \to k^{C_n} \to H \to k\mathbb S_{n-1} \to k$.

The results in this section are a special case of those obtained in Section \ref{almost-simple}. In particular, Proposition \ref{simetrico} follows from Proposition \ref{sym-2}.

Let  $\C(K_1, K_2, B)$ be the fusion subcategory of $D(\mathbb S_n)$
corresponding to the triple $(K_1, K_2, B)$ satisfying (a'), (b') and such that  $|K_1| = |K_2|$.

Suppose $n \geq 5$. So that the only normal
subgroups of $\mathbb S_n$ are $1$, $\mathbb A_n$ and $\mathbb
S_n$.

Therefore  the only fusion subcategory of
$D(\mathbb S_n)$ of dimension $n!$ is $\C(1; 1; 1) \simeq k\mathbb
S_n$.

Combining this with Corollary \ref{h-triple} we get the following:

\begin{lemma}\label{sym} Suppose $n \geq 5$.

\begin{enumerate}\item If $H$ is quasitriangular, then $H$ is a twisting of
$k\mathbb S_n$.

\item If $H^*$ is quasitriangular, then $H^*$ is a twisting of $k\mathbb S_n$.
\end{enumerate}\end{lemma}

\begin{proof} If $H$ is quasitriangular, then $\Rep H$ is a fusion subcategory of $\Rep D(H) \simeq \Rep D(\mathbb S_n)$ of dimension $n!$.
Hence $\Rep H \simeq \Rep \mathbb S_n$. The other claim is
similar, because the same arguments apply as well to $H^* \simeq
k^{\mathbb S_{n-1}} \# kC_n$, and $D(H) \simeq D(H^{*\cop})$.
\end{proof}

\begin{proposition}\label{simetrico} Suppose $n \geq 5$. Then neither $H$ nor $H^*$
admit a quasitriangular structure. \end{proposition}

\begin{proof} Suppose on the contrary that $H$ admits a
quasitriangular structure. Then $H$ would be a twisting of the
group algebra $k\mathbb S_n$, in view of Lemma \ref{sym}.
Similarly, if $H^*$ admits a quasitriangular structure, then $H^*$
would be a twisting of $k\mathbb S_n$.

Note that twisting preserves Hopf algebra quotients. Consider
first the case of the dual Hopf algebra $H^*$. In this case, there
is a quotient Hopf algebra $H^* \to kC_n$, which gives rise to a
quotient Hopf algebra $k\mathbb S_n \to L$, where $L$ is a certain
twisting of $kC_n$. Thus $L \simeq kC_n$. But this gives $n$
non-equivalent representations of dimension $1$ of the group
algebra $k\mathbb S_n$, which is impossible because $n > 2$. Hence
$H^*$ admits no quasitriangular structure.

Similarly, if $H$ admits a quasitriangular structure, then $H$
must be a twisting of $k\mathbb S_n$. But $H$ has a quotient Hopf
algebra isomorphic to $k\mathbb S_{n-1}$. This gives a quotient
Hopf algebra $k\mathbb S_n \to L$, with $\dim L = (n-1)!$. On the
other hand, every such quotient Hopf algebra is of the form $kF$,
where $F$ is a quotient group of $\mathbb S_n$. Then necessarily
$|F| = 2 = (n-1)!$, which contradicts the assumption that $n \geq
5$. This contradiction shows that $H$ does not admit a
quasitriangular structure. The proof of the proposition is now
complete. \end{proof}

\begin{remark} If $n = 3$, then $H \simeq k\mathbb S_3$ is a
group algebra. Thus $H$ is quasitriangular in this case, while $J
= H^* \simeq k^{\mathbb S_3}$ is not.

The arguments fail if $n = 4$ because we could also take $(K_1, K_1,
B)$, where $K_1 \subseteq \mathbb S_4$ is the Klein subgroup, which
is abelian.
\end{remark}

\section{Hopf algebras associated to simple groups}\label{simple}

Let $G = F\Gamma$ be a proper exact factorization of the group $G$. We
assume in this section that the group $G$ is simple non-abelian.

Let $H$ be a semisimple Hopf algebra fitting into an exact
sequence \begin{equation}\label{sec}k \to k^{\Gamma} \to H \to kF
\to k, \end{equation} corresponding to the matched pair $(F,
\Gamma)$.

In particular, $H$ is isomophic to a bicrossed product $H \simeq
k^{\Gamma} {}^{\tau}\#_{\sigma} kF$, where $(\sigma, \tau)$
represents the element of the group $\Opext(k^{\Gamma}, kF)$
determined by the extension \eqref{sec}.

\medbreak In view of Theorem \ref{kac}  there is an equivalence of
fusion categories $\Rep D(H) \simeq \Rep D^{\omega}(G)$, where
$\omega \in H^3(G, k^{\times})$ is the $3$-cocycle correspondig to
the class $[\sigma, \tau] \in \Opext(k^{\Gamma}, kF)$ under the
Kac exact sequence \ref{kac-es}.

\begin{lemma}\label{twisting} Suppose $H$ admits a quasitriangular structure. Then
$H$ is twist equivalent to the group algebra $kG$. \end{lemma}

\begin{proof} Under these assumptions, $\Rep H$ is isomorphic as a
fusion category to a full fusion subcategory of $\Rep D(H) \simeq
\Rep D^{\omega}(G)$.

In view of the description of such fusion subcategories in \ref{nnw}, and since the group $G$ is
simple by assumption, the only fusion subcategory of $\Rep
D^{\omega}(G)$ of dimension $|G|$ is $\C(1, 1, 1) \simeq \Rep G$.
Therefore we must have $\Rep H \simeq \Rep G$, which implies the
lemma. \end{proof}

\begin{theorem}\label{simp} Let $H$ be a Hopf algebra fitting into an
exact sequence \eqref{sec}, where $\Gamma$ and $F$ are proper
subgroups of $G$. Then $H$ admits no quasitriangular structure.
\end{theorem}

\begin{proof} Suppose on the contrary that $H$ admits a
quasitriangular structure. Then $H$ is twist equivalent to $kG$,
by Lemma \ref{twisting}. On the other hand, $H$ has a Hopf algebra
quotient $H \to kF$ that corresponds, because twisting preserves
quotients, to a Hopf algebra quotient $kG \to L$, with $1 < \dim L
= |F| < |G|$. In particular, $L \simeq k\overline G$, where
$\overline G$ is a quotient of the group $G$. This is a
contradiction, since the group $G$ is simple by assumption. Thus
$H$ admits no quasitriangular structure, as claimed. \end{proof}

\begin{remark} We have shown, as part of the proof of Proposition \ref{simetrico} and Theorem \ref{simp}, that the bicrossed product Hopf algebra $H$ is not twist equivalent to the group algebra of $G$. For the case where $G = \mathbb S_n$, $n = p+1$ or $p+2$, $p > 3$ a prime number, and $H$ corresponds to the exact factorization considered in Section \ref{symm}, this fact follows from the main result of \cite{collins} that says that $H$ is not isomorphic as an algebra to any group algebra. The analogous result is true for a bismash product (split extension) Hopf algebra associated to the groups ${\rm PGL}_2(q)$, $q \neq 2, 3$, as shown in \cite{clarke}. \end{remark}

\section{Hopf algebras associated to almost simple groups}\label{almost-simple}

Let $G$ be a finite group. Recall that $G$ is called \emph{almost simple} if it is isomorphic to a group $\tilde G$ (that we shall identify with $G$ in what follows) such that $N \leq \tilde G \leq \Aut N$, for some non-abelian finite simple group $N$.
In this case, the \emph{socle} of $G$, that is, the subgroup generated by the minimal normal subgroups of $G$, coincides with $N$.
Furthermore, a finite group $G$ is almost simple if and only if its socle is a simple non-abelian group.

In particular, we have the following

\begin{lemma}\label{normal-sub} Let $N \leq G \leq \Aut (N)$, where $N$ is simple non-abelian. Then for every normal subgroup $1 \neq K \trianglelefteq G$, we have $N \leq K$.\qed \end{lemma}

\medbreak Examples of almost simple groups are the non-abelian simple groups. Also,  the symmetric groups $\mathbb S_n$, are almost simple, for all $n \geq 5$. Indeed, $\mathbb S_n \simeq \Aut(\mathbb A_n)$, for all $n \neq 6$, and $\Aut(\mathbb A_6)/\mathbb S_6 \simeq \mathbb Z_2$.

\medbreak In view of the Classification Theorem of finite simple groups, the simple non-abelian group $N$ is isomorphic to exactly one of the following groups:

(i) the alternating group $\mathbb A_n$, $n \geq 5$,

(ii) a group of Lie type, or

(iii) one of the twenty-six sporadic simple groups.

\medbreak These groups, together with their orders,  are listed,
for instance, in \cite{GLS}. We shall make use of the order of the
outer automorphism groups from \cite[2.1]{LPS-1}.

\begin{theorem}\label{alm-simp} Let $G$ be an almost simple group and let $G = F\Gamma$ be a proper exact factorization of $G$. Let $H$ be a Hopf algebra fitting into an
exact sequence \eqref{sec}. Then $H$ admits no quasitriangular structure.
\end{theorem}

\begin{proof} By Theorem \ref{kac} we have an
equivalence of fusion categories $\Rep D(H) \simeq \Rep
D^{\omega}(G)$, where $\omega \in H^3(G, k^{\times})$ is the
$3$-cocycle correspondig to the class $[\sigma, \tau] \in
\Opext(k^{\Gamma}, kF)$ under the Kac exact sequence \ref{kac-es}.

Suppose on the contrary that $H$ admits a quasitriangular structure. Then $\Rep H$ is isomorphic as a
fusion category to a full fusion subcategory of $\Rep D(H) \simeq
\Rep D^{\omega}(G)$.

In view of the description of such subcategories in \ref{nnw}, and
since every nontrivial normal subgroup of $G$ contains $N$, which
is not abelian by assumption, the only fusion subcategory of $\Rep
D^{\omega}(G)$ of dimension $|G|$ is $\C(1, 1, 1) \simeq \Rep G$.
Therefore we must have $\Rep H \simeq \Rep G$, which implies $H$
is twist equivalent to the group algebra $kG$. In particular $H$
is isomorphic to $kG$ as an algebra and therefore they have the
same irreducible degrees.

\medbreak The (proper) quotient Hopf algebra $H \to kF$ corresponds to a surjective group homomorphism $f: G \to T$ for some group $T$ such that $kT$ is twist equivalent to $kF$, so that $|T| = |F|$. Thus, in view of Remark \ref{div-f}, we get:
\begin{equation}\label{condition}\dim V \, \text{ divides } \; |F| = |T|, \, \text{ for all irreducible representation } V \text{ of } G.\end{equation}

Letting  $L \subseteq G$ be the kernel of $f$, and since $L$ is a nontrivial normal subgroup of $G$, Lemma \ref{normal-sub} implies that $N \subseteq L$. In particular, $T$ is isomorphic to a subgroup of $G/N \leq \Out(N)$.

Since the group $G$ is an extension of $N$, $N$ being a normal subgroup, then the irreducible representations of $N$ should also satisfy Condition \eqref{condition}.

\medbreak We shall show next that such a subgroup $T$ cannot exist, thus proving the theorem. The proof will go case by case, considering the different possibilities for the simple non-abelian group $N$ listed above.

\medbreak \emph{Case (i). $N \simeq \mathbb A_n$, $n \geq 5$.} If $n \neq 6$, then $G = \mathbb A_n$ or $\mathbb S_n$ and $\Out(N) = 2$, hence $|T| = 1$ or $2$. Combined with Condition \eqref{condition}, this implies a contradiction. If $n = 6$, then $\Out(N) \simeq \mathbb Z_2 \times \mathbb Z_2$, which gives similarly a contradiction ($\mathbb A_6$, hence also $G$, has irreducible representations of degree $> 4$).

\medbreak \emph{Case (ii). $N$ is a group of Lie type over a
finite field of characteristic $p$.} In this case, the Steinberg
representation is an irreducible representation of $N$ of degree
equal to the largest power of $p$ dividing the order of the group
\cite{carter, humphreys}.  That this representation does not
satisfy Condition \eqref{condition},  follows  by inspection of
the order of the group $\Out(N)$ in \cite[Table 2.1 A and
B]{LPS-1}.

\medbreak \emph{Case (iii). $N$ is a sporadic simple group.} In this case, $|\Out (N)| = 1$ or $2$. The same argument as in Case (i) discards this possibility. \end{proof}

As a consequence, we get the following statement,
generalizing Proposition \ref{simetrico}.

\begin{proposition}\label{sym-2} Let $H$ be a Hopf algebra fitting into an
exact sequence \eqref{sec}, where $\Gamma$ and $F$ are proper
subgroups of $\mathbb S_n$, $n \geq 5$. Then $H$ admits no
quasitriangular structure. \qed
\end{proposition}

\bibliographystyle{amsalpha}

\end{document}